\documentclass[12pt]{article}

\usepackage{amssymb,amsmath,amsthm,url}
\newtheorem{theorem}{Theorem}[section]
\newtheorem{cor}[theorem]{Corollary}
\newtheorem{lemma}[theorem]{Lemma}
\newtheorem{prop}[theorem]{Proposition}

\usepackage{tikz}

\usepackage{enumitem}

\newtheorem{preconj}{Conjecture}[section]
\newenvironment{conj}{\preconj\rm}{\endpreconj}
\newtheorem{preex}[preconj]{Example}
\newenvironment{ex}{\preex\rm}{\endpreex}
\newtheorem{preqn}[preconj]{Question}
\newenvironment{qn}{\preqn\rm}{\endpreqn}
\newtheorem{prerk}[preconj]{Remark}
\newenvironment{rk}{\prerk\rm}{\endprerk}

\newcommand{\pgl}{\mathop{\mathrm{PGL}}}

\newcommand{\gl}{\mathop{\mathrm{GL}}}
\newcommand{\gf}{\mathop{\mathrm{GF}}}

\newcommand{\psl}{\mathop{\mathrm{PSL}}}
\newcommand{\SL}{\mathop{\mathrm{SL}}}

\newcommand{\psu}{\mathop{\mathrm{PSU}}}
\newcommand{\pgaml}{\mathop{\mathrm{P}\Gamma\mathrm{L}}}
\newcommand{\Aut}{\mathop{\mathrm{Aut}}}

\newcommand{\Out}{\mathop{\mathrm{Out}}}

\begin{document}
\title{Groups generated by derangements}
\author{R. A. Bailey\footnote{School of Mathematics and Statistics, University of St Andrews, St Andrews, Fife KY16 9SS, UK},
Peter J. Cameron\footnotemark[\value{footnote}],\\
Michael Giudici\footnote{Centre for the Mathematics of Symmetry and Computation, University of Western Australia, Crawley, WA 6009, Australia}
and Gordon F. Royle\footnotemark[\value{footnote}]}
\date{}
\maketitle

\begin{abstract}
We examine the subgroup $D(G)$ of a transitive permutation group $G$ which is
generated by the derangements in $G$. Our main results bound the index of
this subgroup: we conjecture that, if $G$ has degree $n$ and is not a Frobenius
group, then $|G:D(G)|\leqslant\sqrt{n}-1$; we prove this except when $G$ is a
primitive affine group. For affine groups, we translate our conjecture into an
equivalent form regarding $|H:R(H)|$, where $H$ is a linear group on a finite
vector space and $R(H)$ is the subgroup of $H$ generated by elements having 
eigenvalue~$1$. 

If $G$ is a Frobenius group, then $D(G)$ is the Frobenius kernel, and so 
$G/D(G)$ is isomorphic to a Frobenius complement. We give some examples
where $D(G)\ne G$, and examine the group-theoretic structure of $G/D(G)$; in
particular, we construct groups $G$ in which $G/D(G)$ is not a Frobenius
complement.
\end{abstract}

\section{Introduction}

Jordan proved in 1872 that a finite transitive permutation group $G$ of degree
$n>1$ must contain a derangement (an element with no fixed points). The
existence of such elements is important in various contexts in number theory
and elsewhere \cite{fks,serre,zantema}. It is known that there must be
many derangements (at least $|G|/n$, see \cite{cc}), and that at least one
has prime power order \cite{fks}. We are interested here in the subgroup
$D(G)$ of $G$ generated by the derangements in $G$.

In most cases, $D(G)=G$. For example, of the $3302368$ transitive groups of
degree from $2$ to $47$ inclusive as classified in \cite{HR} and available in \textsc{Magma} \cite{Magma}, only $892$ have $D(G)\ne G$ (of which $103$
are Frobenius groups); and, of the $24558$ primitive groups of degree from $2$ to
$4095$ inclusive as classified in \cite{CQRD} and available in \textsc{Magma}, only $9155$ have $D(G)\ne G$ (of which $7872$ are Frobenius groups).

The question was first considered by H.~Zantema~\cite{zantema}, who proved
the first two parts of the following theorem. We include the proof since we
extend the ideas to prove the rest of the theorem.

\begin{theorem}\label{thm1}
Let $G$ be a transitive permutation group on $\Omega$, and $N=D(G)$ the
(normal) subgroup generated by the derangements in $G$. Then
\begin{enumerate}\itemsep0pt
\item $N$ is transitive.
\item $N$ contains every element of $G$ whose number of fixed points is
different from $1$.
\item If $r_G$ and $r_N$ denote the permutation ranks of $G$ and $N$, then
\[r_N-1 = (r_G-1)|G:N|.\]
\item The $N$-orbits on ordered pairs of distinct elements are permuted
semiregularly by $G/N$; equivalently, for $\alpha\in\Omega$, the
$N_\alpha$-orbits different from $\{\alpha\}$ are permuted semiregularly
by~$G_\alpha/N_\alpha$.
\end{enumerate}
\label{t:basic}
\end{theorem}

Any Frobenius group $G$ gives an example with $D(G)\ne G$; for in this case
$D(G)$ is the Frobenius kernel, and its index is the order of a point
stabiliser. (This corresponds to the case in Theorem~\ref{t:basic} where
$N_\alpha=\{1\}$.) So, in a sharply $2$-transitive group of degree $n$, we have
$|G:D(G)|=n-1$. On the other hand, by part (d) of the theorem, the index
cannot be larger than $n-1$ (and indeed divides $n-1$), where $n=|\Omega|$.
Equality implies that
$r_G=2$ (so that $G$ is $2$-transitive), and $r_N=n$ (so that $N$ is regular,
and $G$ is a Frobenius group). So:

\begin{cor}\label{cor1}
If $G$ is a transitive permutation group of degree $n>1$, then $|G:D(G)|$
divides $n-1$; equality is possible if and only if $n$ is a prime power.
\end{cor}

We also obtain the following corollary.

\begin{cor}
Let $G$ be a transitive permutation group, and suppose that $D(G)\ne G$. Let
$G_\alpha$ be the stabiliser of $\alpha$, acting on the remaining points.
Then at least half the elements of $G_\alpha$ are derangements, and 
$G_\alpha=D(G_\alpha)$.
\label{c:stab}
\end{cor}

It follows, for example, that if $G$ is a Zassenhaus group (a $2$-transitive
group in which the point stabiliser is a Frobenius group) then $D(G)=G$.

\medskip

Our main interest is in proving better bounds in the case when $G$ is not a
Frobenius group. We prove the following two theorems:

\begin{theorem}\label{imprim}
If $G$ is a transitive imprimitive permutation group of degree $n$, then 
$|G:D(G)|\leqslant\sqrt{n}-1$. Equality is possible if $n$ is an even power
of a prime.
\label{t:imprim}
\end{theorem}

\begin{theorem}
If $G$ is a primitive permutation group of degree $n$ which is not of affine
type, then $|G:D(G)|\leqslant\sqrt{n}-1$.
\label{t:prim}
\end{theorem}

We conjecture that the same  bound is true for all primitive groups which are
not Frobenius groups:

\begin{conj}
If $G$ is a primitive permutation group of degree $n$ which is not a Frobenius
group, then $|G:D(G)|\leqslant\sqrt{n}-1$; moreover, this bound is attained
only if $G$ is an affine group.
\end{conj}

For the first part of this conjecture, it suffices to consider affine groups,
and we explain in Section~\ref{s:affine} the partial results we have obtained
on this. The second part follows from the first together with our results on
non-affine primitive groups, where we obtain substantially better bounds in
all cases. For example, groups of twisted wreath product type
satisfy $D(G)=G$, and almost simple groups with $D(G)\ne G$ can be completely
classified. See Section~\ref{s:prim} below.

Another question we pose is the following:

\begin{qn}
Which groups can arise as $G/D(G)$ for some transitive permutation group $G$?
\end{qn}

We have no example of a group $H$ which cannot be isomorphic to $G/D(G)$ for
any transitive finite permutation group $G$, but the evidence is far too
thin to support the conjecture that all groups arise.

If $G$ is a Frobenius group, then $D(G)$ is the Frobenius kernel, and so
$G/D(G)$ is isomorphic to the Frobenius complement. The structure of Frobenius
complements was determined by Zassenhaus; either such
a group is metacyclic, or it has a normal subgroup of index at most two which
is isomorphic to the direct product of $\SL(2,3)$ or $\SL(2,5)$ and a metacyclic
group. See Passman~\cite{passman} for an account of this.

There are transitive groups with $G/D(G)$ not isomorphic to a Frobenius
complement, though they are rather rare. The smallest degree of a primitive
group with this property is $625$; there are primitive groups of this degree
for which $G/D(G)$ is isomorphic to the Klein group $V_4$ or the symmetric
group $S_3$. In the final section of the paper, we construct a number of
further examples of this phenomenon.

\section{Proofs of the basic results}

We begin with the proof of Theorem~\ref{t:basic}. As noted, parts (a) and (b)
are due to H.~Zantema~\cite{zantema}, and are repeated here since we will
push the arguments a little further to prove the rest of the theorem.

\begin{proof} Let $\pi$ be the permutation character. Since $G$ is
transitive, the Orbit-Counting Lemma gives
\[\sum_{g\in G}(\pi(g)-1) = 0.\]
Now similarly
\[\sum_{g\in N}(\pi(g)-1) = (k-1)|N|,\]
where $k$ is the number of $N$-orbits. So
\[\sum_{g\in G\setminus N}(\pi(g)-1) = -(k-1)|N|.\]
But every term in the sum on the left is non-negative, since all the
elements with $\pi(g)-1<0$ lie in $N$. We conclude that both sides are zero.
The right-hand side shows that $k=1$,
and the left-hand side contains no terms with $\pi(g)>1$, so all such elements
lie in $N$. This proves (a) and (b).

For (c), note that
\begin{eqnarray*}
|G|(r_G-1) &=& \sum_{g\in G}(\pi(g)^2-1),\\
|N|(r_N-1) &=& \sum_{g\in N}(\pi(g)^2-1).
\end{eqnarray*}
Since every element of $G\setminus N$ has $\pi(g)=1$, the two displayed
expressions are equal, which proves (c).

Finally, (d) follows from (c) since the $r_N-1$ orbits of $N$ on ordered pairs
of distinct elements fall into $r_G-1$ orbits under the action of $G/N$.
\end{proof}

We mention another derivation of (b) from (a), since we will need this later.
This depends on the following (well-known) generalisation of the Orbit-Counting
Lemma. For completeness, we give the proof.

\begin{lemma}
Let $G$ be finite transitive permutation group on $\Omega$, and $t$ an arbitrary
permutation on $\Omega$. Then the average number of fixed points of elements
in the coset $tG$ is $1$.
\label{l:coset}
\end{lemma}

\begin{proof}
We follow the usual proof of the Orbit-Counting Lemma. If $G$ is transitive
on $\Omega$, with $|\Omega|=n$, count pairs $(\alpha, g)$ for which
$\alpha\in\Omega$, $g\in G$, and
$\alpha tg=\alpha$. For each of the $n$ choices of $\alpha$, there are
$|G|/n$ elements $g\in G$ mapping $\alpha t$ to $\alpha$; so there are $|G|$ such pairs. Counting the other way, we sum the numbers of fixed points of
elements in the coset $tG$.
\end{proof}

Now suppose that $g\in G\setminus D(G)$. By (a) and Lemma~\ref{l:coset},
the average number of fixed points of elements of $gD(G)$ is $1$, but none
of these elements is a derangement; so all have exactly one fixed point.

\paragraph{Proof of Corollary~\ref{c:stab}} Since $D(G)$ is transitive,
$|G_\alpha:G_\alpha\cap D(G)|=|G:D(G)|>1$. But all the elements of $G_\alpha$
not in $D(G)$ are derangements (they fix only $\alpha$); so there are at least
$|G_\alpha|/2$ derangements in $G_\alpha$, and they generate $G_\alpha$ (since
any group is generated by the complement of any proper subgroup).\qed

\paragraph{Proof of Theorem~\ref{t:imprim}}
Let $N=D(G)$, and $H=G/N$. By Corollary \ref{cor1} we have that $|H|$ divides $n-1$. Moreover, as $N$ is transitive we have $|H|=|G:N|=|G_\alpha:N_\alpha|$. Furthermore, by Theorem \ref{thm1}(d), $G_\alpha/N_\alpha$ permutes the $N_\alpha$-orbits  different from $\{\alpha\}$  semiregularly.

Suppose that $G$ is imprimitive, with $\ell$ blocks of size $k$, where $k\ell=n$.
Then $G_\alpha$ permutes among themselves the $N_\alpha$-orbits in the block
containing $\alpha$; so $|H|$ divides $k-1$. Then also $|H|$ divides
$n-k=k(\ell-1)$, and since $|H|$ is coprime to $k$, we see that $|H|$ divides
$\ell-1$. But $\min\{k,\ell\}\leqslant\sqrt{n}$, and so the result follows.

Equality can be attained if $n$ is a prime power (and a square). Let $V$ be a
$2$-dimensional vector space over the finite field $F$. Then the semi-direct
product of the additive group of $V$ and the multiplicative group of $F$ is
a Frobenius group of order $|F|^2(|F|-1)$. \qed

\section{Affine groups}
\label{s:affine}

In this section we consider affine groups.

\subsection{Preliminaries and a conjecture}

Let $V$ be a $d$-dimensional vector space over the field of order $q$. Let
$T$ be the translation group of $V$, and $H$ a linear group on $V$ (a
subgroup of $\gl(d,q)$). Then the semidirect product $G=T\rtimes H$ is
a transitive permutation group on $V$; it is primitive if and only if the
linear group $H$ is irreducible.

Given a linear group $H$, we let $R(H)$ be the subgroup of $H$ generated by
elements which have an eigenvalue~$1$ in their action on $V$.

\begin{prop}
With the above notation, $D(G)$ is the semidirect product $T\rtimes R(H)$, and
so $|G:D(G)|=|H:R(H)|$ and $G/D(G)\cong H/R(H)$.
\label{p:linear}
\end{prop}

\begin{proof}
Clearly $T\leqslant D(G)$. By Lemma~\ref{l:coset}, the average number of
fixed points of elements in a coset $hT$ (for $h\in H$) is $1$; so there are
two possibilities:
\begin{itemize}
\item some element of $hT$ is a derangement, in which case $hT\subseteq D(G)$
and $h\in D(G)$;
\item every element of $hT$ has exactly one fixed point; then $h$ fixes the
zero vector and no other, so no eigenvalue of $h$ is equal to $1$.
\end{itemize}
So $hT\subseteq D(G)$ if and only if $h\in R(H)$, and the result follows.
\end{proof}

Thus using Theorems \ref{thm1} and  \ref{imprim} we can formulate a result and a conjecture which if true would settle our
main conjecture for primitive groups.

\begin{prop}\label{p:RH}
If $H$ is any subgroup of $\gl(d,q)$, then $|H:R(H)|\leqslant q^d-1$, and
$H/R(H)$ permutes the $R(H)$-orbits semiregularly. If $H$ is reducible, then
$|H:R(H)|\leqslant q^{d/2}-1$.
\end{prop}

\begin{conj}
If $H$ is an irreducible subgroup of $\gl(d,q)$, then either $H$
acts semiregularly on the non-zero vectors of $V$, or
$|H:R(H)|\leqslant q^{d/2}-1$.
\label{c:linear}
\end{conj}

\subsection{An example}

In this subsection, we give an example to show that the bound 
$|G:D(G)|\leqslant \sqrt{n}-1$, if true, is best possible for primitive groups
which are not Frobenius groups, by giving an example meeting the bound.

Let $q$ be a prime power, and $G$ the group
\[\{x\mapsto ax^i+c\mid a,c\in F, a\ne0, i\in\{1,q\}\}\]
of permutations of the field $F$ of order $q^2$. 

Let $A = \{a \in F \mid a^{q+1} = 1\}$, and let $H$ be the subgroup of $G$
consisting of the transformations of the above form with $a \in A$. 
Notice that $A$ is the set of $(q-1)^{\mathrm st}$ powers of non-zero elements of $F$.

Clearly, the group $T=\{x\mapsto x+c:c\in F\}$ of translations is contained in
$D(G)$. Now consider the map $x\mapsto ax^q$.  The point $x$ is fixed if and only if
$x=0$ or $x^{-(q-1)}=a$. If $a^{q+1}=1$, then the equation $x^{-(q-1)}=1$ has
$q-1$ solutions, and so by Theorem \ref{thm1}(b) the map $x\mapsto ax^q$ belongs to $D(G)$. Composing this with
the element $x\mapsto x^q$ (which is in $D(G)$) we see that the map $x\mapsto ax$ also lies in $D(G)$. Thus $H\leqslant D(G)$.

We now consider the transformations not in $H$. Now separately consider transformations of the form $x \mapsto a x + b$ and $x \mapsto a x^q + b$, where in both cases $a \notin A$. In the former case, it is easy to see that $x \mapsto a x + b$ has a unique
fixed point, namely $x = b/(1-a)$, for all $a\ne 1$, and in particular for all
$a \notin A$.
In the latter case, as there are no non-zero solutions
to the equation $x = a x^q$, the transformation $x \mapsto x - a x^q$ has trivial kernel and therefore is surjective. 
In particular, there is a unique value of $x$ such that $x - ax^q = b$ and thus a unique fixed point for the transformation $x \mapsto a x^q + b$. 
Hence every transformation outside $H$ has a unique fixed point, and so $H$ contains all derangements. Thus $H$ contains all the derangements and hence $D(G)\leqslant H$. As we have already seen that $H\leqslant D(G)$, equality holds. It is then clear that $G/D(G)$ has order $q-1$.

\section{Examples}

In this section, we describe a few examples of non-affine groups $G$ with
$D(G)\ne G$. Further affine examples appear in the final section.

There is no useful product construction. For suppose that $G_1$ and $G_2$ are
transitive on $\Omega_1$ and $\Omega_2$, and consider $G_1\times G_2$ acting
on $\Omega_1\times \Omega_2$. Then an element $(g_1,g_2) \in G_1 \times G_2$ is
a derangement if and only if either $g_1$ or $g_2$ is a derangement. So
$D(G_1\times G_2)$ contains both $D(G_1)\times G_2$ and $G_1\times D(G_2)$,
and hence it is equal to $G_1\times G_2$.

\subsection{General remarks}

Before giving some more examples we note a couple of  useful lemmas.

\begin{lemma}\label{lem:socle}
Let $G$ be a primitive permutation group with socle $N$. Then $N\leqslant D(G)$.
\end{lemma}
\begin{proof}
If $N$ is the unique minimal normal subgroup of $G$ then clearly
$N\leqslant D(G)$, as $D(G)\neq 1$ by Jordan's result. If $N$ is not the unique minimal normal of $G$ then by
a well-known ``folklore'' result (see~\cite[Theorem 4.4]{c:pg}),
$N=M_1\times M_2$, where $M_1$ and $M_2$ are regular. Hence we also have
$N\leqslant D(G)$ in this case as well.
\end{proof}

\begin{lemma}\label{lem:coprimecent}
Let $G=N\rtimes \langle \sigma\rangle$ be a permutation group such that $\sigma$ has order a power of the prime  $p$, with $p$ coprime to $|N|$. If $C_G(\sigma)\leqslant G_\alpha$ then $D(G)\leqslant N$.
\end{lemma}

\begin{proof}
Let $g\in G\backslash N$. If $g$ has order a power of $p$ then Sylow's Theorem implies that $g$ is conjugate to an element of $\langle \sigma\rangle$ and hence fixes a point of $\Omega$. Suppose that $g$ does not have order a power of $p$. Then $|g|=mp^i$ for some $i>0$ and with $\gcd(m,p)=1$. Thus there exist $a,b\in\mathbb{Z}$ such that $am+bp^i=1$ and so $g=(g^{p^i})^b(g^m)^a$.  Now we have written $g$ as the product of two commuting elements, one of which (namely $(g^m)^a$) has order a nontrivial power of $p$. Thus $g$ is conjugate to an element of the form $x\sigma^i$ for some $x\in C_G(\sigma)$. Hence $g$ is conjugate to an element of $C_G(\sigma)$ and so fixes a point. Thus all derangements in $G$ lie in $N$.
\end{proof}

\subsection{The examples}\label{sec:eg}

\paragraph{Almost simple groups}
\begin{enumerate}
\item
Let $G=\pgaml(2,2^p)=N\rtimes\langle\sigma\rangle$, where $N=\pgl(2,2^p)$ for
$p$ an odd prime, and $\sigma$ a field automorphism of order $p$, acting on
the set $\Delta$ of right cosets of a subgroup
$H=C_{2^p+1}\rtimes C_{2p}\geqslant C_G(\sigma)$
of index $2^{p-1}(2^p-1)$.  When $p=3$, a \textsc{Magma} calculation shows that
$D(G)=\pgl(2,2^p)$. For $p\geqslant 5$ we have that $p$ is coprime to
$|\pgl(2,2^p)|$ and so Lemma \ref{lem:coprimecent} implies that
$D(G)=\pgl(2,2^p)$.  Thus for all primes $p$ we have $|G:D(G)|=p$.

\item Let $G=\psl(d,p^f)\rtimes \langle \varphi\rangle$ where $f$ is a power of a prime $r$ which does not divide $|\psl(d,p^f)|$, and $\varphi$ is a field automorphism of $\psl(d,p^f)$ of order $f$. Let $H=\psl(d,p)\times  \langle \varphi\rangle$  and let $G$ act on the set of right cosets of $H$. Then by Lemma \ref{lem:coprimecent} we have that $D(G)=\psl(d,p^f)$. (The fact that all derangements in $G$ lie in $\psl(d,p^f)$ was previously observed in~\cite{gms}.)

\end{enumerate}

\paragraph{Product action}

Let $N$ be $\pgl(2,2^p)$ in the action on $\Delta$ defined in part (a) above,
with $p\geqslant 5$. Let $G=N^p\rtimes\langle g\rangle$ act on
$\Omega=\Delta^p$, where $g=(\sigma,1,\ldots,1)(1,2,\ldots, p)$. Then $g$ has
order $p^2$ and we can choose $\alpha\in\Omega$ such that
$G_\alpha=H^p\rtimes\langle g\rangle$. Moreover, 
$C_G(g)=\{(h,\ldots,h)\mid h\in C_N(\sigma)\}\rtimes\langle g\rangle\leqslant G_\alpha$. Thus 
Lemma \ref{lem:coprimecent} implies that $D(G)=N^p$ and so $|G:D(G)|=p^2$. 

\paragraph{Diagonal action}

Let $T$ be a non-abelian simple group, and $p$ be a prime coprime to $|T|$. Let
$G=T^p\rtimes \langle \sigma\rangle$  where $\sigma$ has order $p$ and
permutes the $p$ simple direct factors of $T^p$, acting on the cosets of
$G_\alpha=\{(t,\ldots,t)\mid t\in T\}\times \langle\sigma\rangle=C_G(\sigma)$.
Then $G$ is a primitive group of diagonal type on a set of size $|T|^{p-1}$.
Any element of $T^p$ that is trivial in all but exactly one of the coordinates
is a derangement and so $T^p\leqslant D(G)$ and then Lemma
\ref{lem:coprimecent} implies that $D(G)=T^p$.

%
%
%
%
%
%

\section{Primitive groups}
\label{s:prim}

We now consider the various types of primitive groups, and prove
Theorem~\ref{t:prim} in all cases. By the O'Nan-Scott Theorem, a primitive group that does not preserve a product structure on $\Omega$ is either almost simple, affine or or diagonal type. See for example \cite{c:pg}.

\subsection{Diagonal type}

We note the following famous result, see \cite[Theorem 1.48]{Gor}.

\begin{lemma}\label{lem:fpf}
Let $T$ be a non-abelian finite simple group and let $\tau\in\Aut(T)$. Then there exists $t\in T\backslash\{1\}$ such that $t^\tau=t$.
\end{lemma}

We also need the following lemma.

\begin{lemma}\label{lem:regnorm}
Let $G$ be a transitive permutation group on $\Omega$ with a regular non-abelian minimal normal subgroup. Then $G=D(G)$.
\end{lemma}

\begin{proof}
Let $N$ be a non-abelian regular minimal normal subgroup of $G$. Then 
$N\cong T^k$ for some non-abelian simple group $T$, and $N\leqslant D(G)$. Note 
that, for $\alpha\in\Omega$, we have $G=N\rtimes G_\alpha$. Moreover, we can 
identify $\Omega$ with $N$ such that, for $\alpha=1_N$, each nontrivial element
of $G_\alpha$ acts as a nontrivial automorphism of $N$. Let $g\in G_\alpha$ and
write $g=(\tau_1,\ldots,\tau_k)\sigma$ where each $\tau_i\in\Aut(T)$ and
$\sigma\in S_k$. Suppose that $(i_1,i_2,\ldots,i_r)$ is a cycle of $\sigma$.
By Lemma \ref{lem:fpf}, there exists $t\in T\backslash\{1\}$ such that
$\tau_{i_1}\tau_{i_2}\ldots\tau_{i_k}$ fixes $t$. Let $t_{i_1}=t$ and for each 
$j\in \{2,\ldots, r\}$ let $t_{i_j}=t^{\tau_{i_1}\ldots\tau_{i_{j-1}}}$. Doing 
this for each cycle of $\sigma$ we construct a nontrivial element 
$\beta=(t_1,\ldots,t_k)\in N$ such that $\beta^g=\beta$. Hence $g$ has at 
least two fixed points and so by Theorem \ref{thm1}(b) we have that 
$g\in D(G)$. Since $G=N\rtimes G_\alpha$ it follows that $G=D(G)$.
\end{proof}

We are now able to obtain a bound for $|G:D(G)|$ when $G$ is primitive of diagonal type.

\begin{lemma}
Let $G$ be primitive of diagonal type and $G\neq D(G)$. Then the socle of $G$
is $N=T^p$ for some non-abelian finite simple group $T$ and some odd prime $p$
not dividing $|T|$, and $G$ induces a cyclic group of prime order on the set of $p$ simple direct factors of $N$. Moreover, $|G:D(G)|=p$.
\end{lemma}
\begin{proof}
Let $N=T^k$ be the socle of $G$ and let $\alpha\in \Omega$. We may assume that $N_\alpha=\{(t,t,\ldots,t)\mid t\in T\}$ and by Lemma \ref{lem:socle} we have $N\leqslant D(G)$. Since $N$ is transitive we have $G=NG_\alpha$. Thus it remains to determine which elements of $G_\alpha$ lie in $D(G)$. 

Let $\pi:G\rightarrow S_k$ be the permutation representation of $G$ on the set of $k$ simple direct factors of $N$. By Lemma \ref{lem:regnorm} we only need to consider the case where $\pi(G)$ is transitive and primitive. Since $G=NG_\alpha$ we have that $\pi(G)=\pi(G_\alpha)$. Now $G_\alpha\leqslant \Aut(T)\times S_k$. Identifying $\Omega$ with the set of cosets of $N_\alpha$ in $N$ we see that for $\tau\in\Aut(T)$ we have $(N_\alpha(t_1,\ldots,t_k))^\tau=N_\alpha(t_1^\tau,\ldots,t_k^\tau)$, while for $\sigma\in S_k$ we have $(N_\alpha(t_1,\ldots,t_k))^\sigma=N_\alpha(t_{1^{\sigma^{-1}}},\ldots,t_{k^{\sigma^{-1}}})$.

  Let $X$ be the preimage in $G_\alpha$ of the stabiliser in $S_k$ of the first entry and let $g=\tau\sigma\in X$ with $\tau\in \Aut(T)$ and $\sigma\in S_k$.  By Lemma \ref{lem:fpf}, there exists $t\in T\backslash \{1\}$ such that $t^\tau=t$.  Then $g$ fixes both the coset $N_\alpha$ and the coset $N_\alpha(t,1,\ldots,1)$. It follows from Theorem \ref{thm1}(b) that $X\leqslant D(G)$. Since $\pi(G)$ is a primitive subgroup of $S_k$, $X$ is a maximal subgroup of $G_\alpha$. Suppose first that $\pi(X)\neq 1$. Then there exists $h\in G_\alpha\backslash X$ such that $h$ fixes the second simple direct factor of $N$. Then $h$ fixes the two distinct cosets $N_\alpha$ and $N_\alpha(1,t,1,\ldots,1)$, where $t\in T$ is fixed by $\tau$. This again implies that $h\in D(G)$ and since $G_\alpha=\langle X,h\rangle$ it follows that $G_\alpha\leqslant D(G)$. Thus $G=D(G)$. Hence if $G\neq D(G)$ then we must have that $\pi(X)=1$, that is, $\pi(G)$ is a regular primitive subgroup of $S_k$. Thus $k$ is a prime, $\pi(G)=C_k$, $D(G)=NX$ and $|G:D(G)|=|G:NX|=k$. 

It remains to show that $k$ is coprime to $|T|$. Suppose to the contrary that
$k$ divides $|T|$. Choose $g\in G\setminus NX$. Without loss of generality,
$g=\tau(1,2,\ldots,k)$. Since $\tau$ and $(1,\ldots,k)$ commute, we can choose
$g$ so that $\tau$ has order a power of $k$ (raising $g$ to a power coprime to
$k$ if necessary). Now we can find $s\in T$ with order $k$ and fixed by $\tau$,
as follows: let $P$ be a Sylow $k$-subgroup of $T\langle\tau\rangle$ containing
$\tau$, and choose $s$ to be an element of order $k$ in $Z(P)\cap T$.

Consider the coset $N_\alpha(s,s^2,\ldots,s^{k-1},1)$. We have
\begin{align*}
(N_\alpha(s,s^2,\ldots,s^{k-1},1))^g&=N_\alpha(1,s^\tau,(s^\tau)^2,\ldots,(s^\tau)^{k-1})\\
  &=N_\alpha(1,s,s^2,\ldots,s^{k-1})\\
  &=N_\alpha(s,s^2,\ldots,s^{k-1},1).
 \end{align*}
   Thus $g$ fixes two elements of $\Omega$ and so by Theorem \ref{thm1}(b) it follows that $g\in D(G)$. Since $G=\langle NX,g\rangle$, it follows that $G=D(G)$, a contradiction. Hence $k$ is coprime to $|T|$.

By the Odd Order Theorem, $k$ is odd.
\end{proof}

\subsection{Product action}

Now we discuss the product action case. By \cite[(2.2)]{kovacs}, we may assume that $G$ is contained in $H\wr K$,
where $H$ is the group induced on one coordinate by its stabiliser in $G$, and $K$
the permutation group induced on the coordinates; thus $n=m^k$, where $m$
and $k$ are the degrees of $H$ and $K$ respectively.

\begin{prop}
With the above hypotheses,
\[|G:D(G)| \leqslant k|H:D(H)|.\]
\label{p:product}
\end{prop}

\begin{proof}
Let $G_1$ be the subgroup of $G$ fixing a coordinate. Then $|G:G_1|=k$, and
there is an epimorphism $\phi:G_1\to H$. Let $G_2=D(H)\phi^{-1}$, so that
$|G_1:G_2|=|H:D(H)|$. So we are done if we can show that $G_2\leqslant D(G)$.

But a generator of $G_2$ has no fixed points on the first coordinate of the
product space, so has no fixed points on the whole space. (If a tuple is
fixed then all its coordinates must be fixed.) The result follows.
\end{proof}

We note that the product action examples given in Section \ref{sec:eg} show that this bound is sharp.

\begin{cor}\label{prodbound}
If the primitive group $G$ is contained in a wreath product action as above,
and $G$ is not a Frobenius group, then $|G:D(G)|\leqslant\sqrt{n}-1$.
\end{cor}

\begin{proof}
We have $|H:D(H)|\mid m-1$. Note that primitivity requires $m>2$.
If $k,m\geqslant3$, then $k(m-1)\leqslant m^{k/2}-1$ except for the cases $k=3$,
$3\leqslant m\le 7$. These cases can be tested by computer, and give no
counterexamples.

Suppose that $k=2$, so that $G\leqslant H\wr C_2$.
If $|H:D(H)|<m-1$, then $|H:D(H)|\leqslant(m-1)/2$, and so $|G:D(G)|\leqslant
m-1$ by Proposition~\ref{p:product}, as required. So we may assume that
$|H:D(H)|=m-1$, so that $H$ is sharply $2$-transitive.  Thus, $H=P\rtimes Q$,
where $P$ is the Frobenius kernel and $Q$ the complement. 

The intersection $K$ of $G$ with the base group of the wreath product is
a subdirect product of two copies of $H$, containing $P\times P$ and
invariant under an interchange of the factors. This is an extension of $R^2$ by
$C$, where $R\geqslant P$ has order $rm$, say, and $C$ is a quotient of $Q$ of
order $(m-1)/r$. So $|K|=m^2(m-1)r$. Now $R^2\leqslant D(G)$, since each
element of one factor can be combined with a derangement in the other to give a
derangement in $G$. So $|G|=2m^2(m-1)r$ and $|D(G)|\ge(rm)^2$, giving
$|G:D(G)|\leqslant 2(m-1)/r$.  So we are done unless $r=1$, in which case $R=P$.

In this case, if $D(G)=P^2$, then it is regular, and so $G$ is a Frobenius
group; if not, then $|D(G)|\geqslant 2m^2$, and so $|G:D(G)|\leqslant m-1$, as
required.
\end{proof}

\subsection{Almost simple type}

We now prove Theorem~\ref{t:prim} for almost simple primitive groups.

\begin{lemma}
Let $G$ be an almost simple primitive permutation group of degree $n$.
Then $|G:D(G)|\leqslant \sqrt{n}-1$.
\end{lemma}
\begin{proof}
If $G$ is almost simple with socle $T$ then, by Lemma \ref{lem:socle},
$T\leqslant D(G)\leqslant G\leqslant \Aut(T)$, so $|G:D(G)|$ is bounded by the
order of the outer automorphism group of $T$.
On the other hand, $n$ is at least the degree $n_0$ of the smallest faithful
permutation representation of $T$.  The outer automorphism group of a sporadic simple group has order at most $2$, while from \cite{atlas} we see that
$n_0\geqslant 11$. Similarly, the outer automorphism group of $A_n$ has order
$2$ unless $n=6$, while $n_0=n$. The values for $n_0$ when $T$ is a group of Lie type are given in \cite[Table 4]{GMPS} and the values for $\Out(T)$ are given in \cite[Tables 5.1A and 5.1B]{KL}. We find that the only simple groups with 
$T$ for which $|\Out(T)|>\sqrt{n_0}-1$ are:
\begin{itemize}\itemsep0pt
\item
$T=A_n$ ($n=5,7,8$), $|\Out(T)|=2$, $n_0=n$;
\item
$T=A_6$, $|\Out(T)|=4$, $n_0=6$;
\item
$T=\psl(3,2)$, $|\Out(T)|=2$, $n_0=7$;
\item
$T=\psl(3,4)$, $|\Out(T)|=12$, $n_0=21$;
\item $T=\psl(2,2^f)$ ($f=3,4,5$), $|\Out(T)|=f$, $n_0=2^f+1$.
\end{itemize}

Thus if $G$ is a counterexample, either $n< 36$ or $T=\psl(3,4)$ and $n<169$. A \textsc{Magma} calculation shows that no such counterexamples exist.
\end{proof}

In this case we can say much more. The memoir by Guralnick, M\"uller and
Saxl~\cite{gms} defines a pair of permutation groups $(X,Y)$ to be
\emph{exceptional} if $Y\lhd X$ and $X$ fixes no non-trivial $Y$-orbit on
ordered pairs. They determine all exceptional pairs where $X$ is almost simple
and $X/Y$ is cyclic. This applies to our situation, since if $D(G)\ne G$ then Theorem \ref{thm1}(d) implies that
$(G,D(G))$ is exceptional. Hence if $D(G)<H\leqslant G$ with $H/D(G)$ cyclic and
$G$ almost simple then $D(G)$ must occur in their list.

\begin{theorem}\cite[Theorem 1.5]{gms}\label{thm:gms}
Let $G$ be a primitive almost simple group of degree $n$ and with socle $T$ such that $D(G)\neq G$. Then one of the following holds:
\begin{enumerate}
\item\label{subfield} $T$ is a group of Lie type and $T_\alpha$ is the centraliser in $T$ of a field automorphism of odd prime order $r$. Moreover, $r$ is not the characteristic of $T$, unless $T=\psl(2,q)$;
\item $T=\psl(2,2^f)$ and $T_\alpha=D_{2(2^f+1)}$ with $f\geqslant 3$ odd;
\item $T=\psl(2,p^f)$ and  $T_\alpha\cong D_{p^f-1}$ with $p$ odd and $f$ even;
\item $T=\psl(2,3^f)$ and $T_\alpha= D_{3^f+1},$ with $f\geqslant 3$ odd;
\item $T=Sz(2^f)$ and $T_\alpha$ is the normaliser of a Sylow 5-subgroup of $T$;
\item $T=\psu(3,2^a)$ with $a>1$ odd  and $T_\alpha$ is the stabiliser in $T$ of a decomposition of the 3-dimensional space into the direct sum of three orthogonal nonsingular 1-spaces.
\end{enumerate}
\end{theorem}

We currently do not know any examples here where $G/D(G)$ is not cyclic.

\section{Affine primitive groups}

As we noted earlier, we have not been able to prove the bound
$|G:D(G)|\leqslant\sqrt{n}-1$ for affine primitive groups which are not
Frobenius. We outline here what we have been able to prove.

Recall that it suffices to show that, if $H\leqslant\gl(d,p)$ for prime $p$
and $H$ is irreducible but not semiregular, then $|H:R(H)|\le p^{d/2}-1$. We
work in greater generality, with a view towards Conjecture~\ref{c:linear}.

So let $H\le\gl(d,q)$ be an irreducible linear group. We distinguish three
cases:
\begin{itemize}[align=left]
\item[Case 1:] $R(H)=1$.
\item[Case 2:] $R(H)>1$ and $R(H)$ is reducible.
\item[Case 3:] $R(H)$ is irreducible.
\end{itemize}

\begin{lemma}
Case 1 occurs if and only if $H$ is semiregular on non-zero elements.
\end{lemma}

\begin{proof}
If $R(H)=1$, then every element of $H\setminus\{1\}$ has no eigenvalues $1$,
and so fixes no non-zero vector. The converse is clear.
\end{proof}

\begin{lemma}
If Case 2 occurs, then $H$ preserves a direct sum or tensor product
decomposition of $V$.
\end{lemma}

\begin{proof}
Let $W$ be a minimal non-zero $R(H)$-invariant subspace. Let
$\mathcal{S}=\{Wg:g\in H\}$. Then every subspace in $\mathcal{S}$ is
$R(H)$-invariant. By minimality, any two members of $\mathcal{S}$ intersect in
$\{0\}$.
Also, the subspace $\langle\mathcal{S}\rangle$ is $H$-invariant. Since $H$ is
irreducible, $\langle\mathcal{S}\rangle=V$.  Note that Proposition \ref{p:RH} implies that $H/R(H)$ permutes $\mathcal{S}$
regularly. Let $\dim(W)=e$.

\subparagraph{Case 2A:} $V=\bigoplus\{W:W\in S\}$. Then $|\mathcal{S}|=d/e$ and
$H$ preserves this direct sum decomposition.

\subparagraph{Case 2B:} $|\mathcal{S}|>d/e$. 

We claim there is a subset of $\mathcal{S}$ whose direct sum is $V$. For
choose a subset of $\mathcal{S}$, say $\mathcal{S}_0$, maximal subject to
generating its direct sum, and suppose $U\in\mathcal{S}\setminus\mathcal{S}_0$.
Let $X$ be the direct sum of the spaces in $\mathcal{S}_0$. Then $X$ is also
$R(H)$-invariant, and so is its intersection with $U$. If
$X\cap U=\{0\}$, then $\mathcal{S}_0\cup\{U\}$ also generates its direct
sum, contrary to assumption. So $U\subseteq X$. But if this holds for all
$U\in\mathcal{S}\setminus\mathcal{S}_0$, then the span of the spaces in
$\mathcal{S}$ is $W$, contradicting the fact that $H$ is irreducible.

Suppose that $V=W_1\oplus\cdots\oplus W_k$, where $W_i\in\mathcal{S}$. If
$W'$ is another subspace in $\mathcal{S}$, then each vector in $W'$ has unique
projections onto at least two $W_i$. Since $R(H)$ fixes all these spaces, we
have $R(H)$-invariant isomorphisms between them.

Now define a relation on $\mathcal{S}$ by the rule that $U_1\sim U_2$ if the
actions of $R(H)$ on $U_1$ and $U_2$ are isomorphic. The result of the
preceding paragraph shows that this relation is not the relation of equality,
and it is clearly an equivalence relation. The span of an equivalence class
is a $R(H)$-invariant subspace, which contains no members of any other
equivalence class. So, arguing as before, $V$ is a direct sum of these
subspaces.

If there is more than one equivalence class, then $H$ preserves this direct
sum decomposition.

If there is just one equivalence class, then $V\cong W\otimes U$ for some
space $U$; and $R(H)$ acts on the first factor of the tensor product. 
\end{proof}

Finally, suppose that Case 3 occurs, so $R(H)$ is an irreducible linear group.
In this case, the obvious approach is to apply Aschbacher's Theorem~\cite{Asch} to $H$. We have dealt with some of the cases, but have not completed the
analysis. We make one simple observation.

\begin{lemma}
The conjecture holds if $H$ is a subfield subgroup or an imprimitive linear
group.
\end{lemma}

\begin{proof}
In the subfield case, suppose that $H\leqslant\gl(d,q_0)\leqslant\gl(d,q)$,
where $q=q_0^e$ with $e>1$. Observing that the eigenvalues of an element of
$H$ are the same whether we regard $H$ as acting on $\gf(q)^d$ or $\gf(q_0)^d$,
we see that  $|H:R(H)|\leqslant q_0^d-1\leqslant q^{d/e}-1$, and the result
follows since $e\geqslant2$.

In the imprimitive case, the semidirect product $T\rtimes H$ is contained in a
wreath product with product action, and the result follows from 
Corollary~\ref{prodbound} (whose proof did not assume that $G$ is not affine).
\end{proof}

\section{On the quotient $G/D(G)$}

In this section we consider the group-theoretic structure of the quotient
$G/D(G)$.

We have seen that any Frobenius complement can occur as this quotient. It turns out that in general the class of groups that
can appear is wider, as the following examples testify.

\begin{ex}\label{eg:small}
\begin{enumerate}
\item Let $X,Y\leqslant \gl(2,5)$ such that $X\cong D_{12}$ and $Y\cong Q_8$. Then $R(X)=D_{12}$, as it is generated by its non-central involutions. Moreover,   any element of $X$ that is not an involution does not have any eigenvalues  in $\gf(5)$. Furthermore, $R(Y)=\langle -I_2\rangle$ and all eigenvalues of elements of $Y$ lie in $\gf(5)$. Let $H=X\circ Y\leqslant \gl(2,5)\circ \gl(2,5)$ acting on the tensor product of two $\gf(5)$-spaces of dimension $2$. (Here $\circ$
denotes central product.) Then all elements of $H$ with 1 as an eigenvalue lie in $X$ and so $H/R(H)=Y/R(Y)\cong C_2^2$. The primitive group $G$ with $G/D(G)\cong C_2^2$ arising from Proposition \ref{p:linear} is the number 41 of degree 625 in the  \textsc{Magma} database.

\item Let $X,Y\leqslant \gl(4,23)$ with $X\cong D_{44}$ and $Y\cong \SL(2,3)$. Both $X$ and  $Y$ are irreducible. Then $R(X)=X$ as it is generated by involutions. Moreover,   all eigenvalues of $X$ lie in $\gf(23)$. The group $Y$ acts semiregularly on the set of 1-dimensional subspaces of $\gf(23)^2$ and so $-I_2$ is the only element of $Y$ with eigenvalues in $\gf(23)$. Let  $H=X\circ Y\leqslant \gl(2,23)\circ \gl(2,23)$ acting on the tensor product of two $\gf(23)$-spaces of dimension 2. Then $H/R(H)=Y/R(Y)\cong A_4$ and so by  Proposition \ref{p:linear} we get a primitive permutation group $G$ of degree $23^4$ with $G/D(G)\cong A_4$.

\item Here take $X,Y\leqslant \gl(2,59)$ with $X\cong D_{116}$ and $Y\cong \SL(2,5)$. The group $Y$ acts semiregularly on the set of 1-dimensional subspaces of $\gf(59)^2$ and so taking $H=X\circ Y\leqslant \gl(2,59)\circ \gl(2,59)$ acting on the tensor product of two $\gf(59)$-spaces of dimension 2 the same argument as above yields a primitive group with $G$ with $G/D(G)\cong A_5$.
\end{enumerate}
\end{ex}

We now give an infinite family of non-Frobenius examples.

\begin{lemma}\label{lem:dihedral}
Let $p$ be a prime and $f\geqslant 1$ such that $q=p^f \equiv -1 \pmod 4$. Then there is a primitive group $G$ such that $G/D(G)\cong D_{q+1}$.
\end{lemma}
\begin{proof}
Let $X$ be the subgroup of $\gl(2,q)$ generated by $\begin{pmatrix} u &0\\0&u^{-1}\end{pmatrix}$ (for $u\ne0$) and $\begin{pmatrix} 0&1\\1&0\end{pmatrix}$. Then $X\cong D_{2(q-1)}$.
Since $q\equiv -1\pmod 4$, the group $X$ does not have an element of order 4. Moreover, $X$ is generated by its non-central involutions and these all have 1 as an eigenvalue. Furthermore, all eigenvalues of the elements of $X$ lie in $\gf(q)$.

Let $x,y\in \gf(q^2)$ with $x$ having order $q+1$ and $y$ having order $2(q+1)$. Consider $x$ and $y$ as elements of $\gl(2,q)$. Let $\sigma$ be the field automorphism of $\gf(q^2)$ that raises each element to its $q^{\mathrm{th}}$ power and consider $\sigma$ as an element of $\gl(2,q)$. Then $(y\sigma)^2=y^{q+1} =-I_2=x^{(q+1)/2}$. Thus $y\sigma$ is an element of order 4. Let $Y=\langle x,y\sigma\rangle\leqslant \gl(2,q)$. All elements in $\langle x\rangle$ other than those in $\langle -I_2\rangle$ have no elements in $\gf(q)$ as an eigenvalue. All elements of $Y$ outside $\langle x\rangle$ have order 4 and the condition on $q$ implies that they also have no eigenvalues in $\gf(q)$. Moreover, $x^{y\sigma}=x^q=x^{-1}$ and so $Y/\langle -I_2\rangle\cong D_{q+1}$.

Now take 
$H= X \circ Y \leqslant \gl(2,q) \circ \gl(2,q)$ acting on the tensor product of two $\gf(q)$-spaces of dimension 2. For each $g\in X$ and $h\in Y$, the eigenvalues of the element arising from $(g,h)$ are of the form $\lambda\mu$ where $\lambda $ is an eigenvalue of $g$ and $\mu$ is an eigenvalue of $h$.  Since the elements of  $Y\backslash\langle -I_2\rangle$ do not have elements of $\gf(q)$ as eigenvalues, the elements of $H$ with 1 as an eigenvalue lie in $X$ and so $H/R(H)= Y/Z(Y) \cong D_{q+1}$. Moreover, as $X$ and $Y$ are both irreducible subgroups of $\gl(2,q)$ we have that $H$ is an irreducible subgroup of $\gl(4,q)$.  Thus by Proposition \ref{p:linear}, there exists a primitive group $G$ such that $G/D(G)\cong D_{q+1}$.
\end{proof}

\medskip

Under extra hypotheses, we can restrict the structure of the quotient. For
example:

\begin{prop}
Suppose that the transitive group $G$ has a regular normal subgroup $N$, and
that $G$ splits over $D(G)$, say $G=D(G)\rtimes H$. Then $N$ is nilpotent
and $H$ is isomorphic to a Frobenius complement.
\end{prop}

\begin{proof}
Non-identity elements of $H$ have unique fixed points. It follows that $H$
fixes a point $\alpha$ and is semiregular on $\Omega\setminus\{\alpha\}$.
(If not, then $H$ acts faithfully as a regular or Frobenius group on each
orbit, and with at least one Frobenius orbit. But then elements of the
Frobenius kernel $K$ can be recognised -- $K$ is the Fitting subgroup of $H$ --
and so they have no fixed points at all, a contradiction.)

Thus $H$ normalises $N$ and acts semiregularly 
on $N\setminus\{1\}$, so that $NH$ is a Frobenius group with kernel $N$ and
complement $H$. Then $N$ is nilpotent by Thompson's theorem.
\end{proof}

We note that for the examples in Example \ref{eg:small} and Lemma \ref{lem:dihedral}, $G$ does not split over $D(G)$.

On the other hand, every Frobenius complement can occur in a non-Frobenius
group:

\begin{prop}
Let $H$ be a Frobenius complement. Then there is a transitive, non-Frobenius
group $G$ such that $G/D(G)\cong H$.
\end{prop}

\begin{proof}
Suppose that $NH$ is a Frobenius group on a set $\Delta$ with kernel $N$ and complement $H$.
Without loss of generality we may suppose that $N$ is abelian. (For by
Thompson's theorem, $N$ is nilpotent; thus $Z(N)\ne\{1\}$, and $H$ acts 
faithfully and fixed-point-freely on $Z(N)$, so $Z(N)H$ is a Frobenius group.)
For convenience we write $N$ additively below.

Choose a prime $q$ which does not divide $|H|$. Let $G=N^q\rtimes(H\times C_q)$ act on $\Delta^q$ in product action,
where $H$ acts in the same way on each factor and $C_q$ permutes the factors.
We have $N^q\leqslant D(G)$. Moreover, elements of $C_q$ fix the diagonal
elements of $N^q$, so by Theorem \ref{thm1}(b) $C_q\leqslant D(G)$. We show
that elements outside $N^q\rtimes C_q$ have just one fixed point; it follows
that $D(G)=N^q\rtimes C_q$, and so $G/D(G)\cong H$ as required. Since $N^q$
is a regular normal subgroup, we can identify $\Omega$ with $N^q$.

Take an element $g=h(a_1,\ldots,a_q)\sigma^i$, where
$C_q=\langle\sigma\rangle$, $a_1,\ldots,a_q\in N$, and $h\ne1$, and suppose
that $g$ fixes $(x_1,\ldots,x_q)$, with $x_1,\ldots,x_q\in\Delta$.
\begin{itemize}\itemsep0pt
\item[Case 1:] $i=0$. Then
\[(x_1,\ldots,x_q)g=(x_1^h+a_1,x_2^h+a_2,\ldots,x_q^h+a_q).\]
So, if $g$ fixes $(x_1,\ldots,x_q)$, we have $x_i^h+a_i=x_i$ for all $i=1,\ldots,q$.
Since $N^q\rtimes H$ is a Frobenius group and $h\ne1$, there is a unique such element.
\item[Case 2:] $i\ne 0$. Without loss of generality, $i=1$. Then
\[(x_1,\ldots,x_q)g=(x_q^h+a_q,x_1^h+a_1,\ldots,x_{q-1}^h+a_{q-1}).\]
So, if $g$ fixes $(x_1,\ldots,x_q)$, then
\[x_1^h+a_1=x_2,x_2^h+a_2=a_3,\ldots,x_q^h+a_q=x_1.\]
Telescoping these formulae gives $x_1^{h^q}+b_1=x_1$, where
\[b=a_1^{h^{q-1}}+\cdots+a_q.\]
Now $q$ is coprime to $|H|$, so $h^q\ne 1$; the same argument as in Case 1
shows that the value of $x_1$ is uniquely determined. A similar argument
shows that $x_2,\ldots,x_q$ are unique.
\end{itemize}
The proof is complete.
\end{proof}

\begin{rk}
In all examples constructed in this section, the group $G/D(G)$, if not
itself a Frobenius complement, is a quotient of one. So we tentatively propose
the following problem:
\end{rk}

\begin{qn}
Is it true that, for any finite transitive permutation group $G$, the group
$G/D(G)$ is a quotient of a Frobenius complement?
\end{qn}

\section{One more problem}

The derangements in a finite transitive permutation group $G$ form a non-empty
union of conjugacy classes; so, if $G$ is simple, they generate $G$.
In a recent preprint, Larsen, Shalev and Tiep \cite{lst} proved the following
theorem:

\begin{theorem}
Let $G$ be a finite simple transitive permutation group. If $|G|$ is
sufficiently large, then any element of $G$ can be written as the product of
two derangements.
\end{theorem}

More generally, we could pose the following problem:

\begin{qn}
Is it possible to classify the finite transitive permutation groups $G$ for
which some element of $D(G)$ cannot be written as the product of two
derangements?
\end{qn}

We note that, in a Frobenius group $G$, every non-identity element of $D(G)$
is a derangement.

\paragraph{Acknowledgment}
This work was begun when the first two authors were visiting The University of
Western Australia in 2016; they acknowledge with thanks support from UWA. The research of the last two authors is supported by the Australian Research Council Discovery Project DP200101951.

The first three authors would like to thank the Isaac Newton Institute for Mathematical Sciences, Cambridge, for support and hospitality during the programme
\textit{Groups, representations and applications: new perspectives},
where part of the research for this paper was done.
This work was supported by EPSRC grant no EP/R014604/1.

\end{document}